\documentclass[10pt,oneside,reqno]{amsart}

\usepackage[colorlinks=true,linkcolor=red,citecolor=blue]{hyperref}
\usepackage{xcolor}
\usepackage[final]{graphicx}
\usepackage{amsfonts}
\usepackage{pdfsync}
\usepackage{amsmath}
\usepackage{amssymb}
\usepackage{amsthm}
\usepackage{caption}
\usepackage{enumerate} 
\usepackage{dcolumn} 
\usepackage{mathrsfs}

\newtheorem{theorem}{Theorem}[section]
\newtheorem{lemma}[theorem]{Lemma}
\newtheorem{corollary}[theorem]{Corollary}
\newtheorem{proposition}[theorem]{Proposition}

\theoremstyle{remark}
\newtheorem{remark}{Remark}

\newcommand{\R}{{\mathord{\mathbb R}}}

\newcommand{\N}{{\mathord{\mathbb N}}}

\newcommand{\di}{\,\mathrm{d}}

\newcommand{\bvec}[1]{\boldsymbol{#1}}

\title[Flat Vlasov-Poisson system with external potential]{Existence and stability of steady states solutions of Flat Vlasov-Poisson system with a central mass density}

\date{\today}

\author[M. Moreno]{Matías Moreno Bustamante}
\email{matias.bustamante.23@ucl.ac.uk}

\begin{document}

\maketitle

\begin{abstract}
   We study a Newtonian model which allows us to describe some extremely flat objects in galactic dynamics. This model is described by a partial differential equation system called Vlasov-Poisson, whose solutions describe the temporal evolution of a collisionless particle system in the phase space, subject to a self-interacting gravitational potential. We treat the Flat Vlasov-Poisson system with an external gravitational potential induced by a fixed mass density. The aim of this article is the study of the existence, regularity, and stability of steady states solutions of the Flat Vlasov-Poisson system in this case. We solved a variational problem to find minimizers for the Casimir-Energy functional in a suitable set of functions. The minimization problem is solved through a reduction of the original optimization problem with a scheme used in \cite{FirtReinI}, but instead of a concentration-compactness argument, we use a symmetrization argument to construct a spherically symmetric solution for the reduced problem. It was proven that this minimizer induces a solution for the original minimization problem. The regularity of the gravitational potential was also obtained, implying that the solutions are steady states of the Flat Vlasov-Poisson system. The minimization problem also works as a key to give us a similar non-linear stability result.
\end{abstract}

\bigskip

\noindent{\bf Acknowledgments.}  
The author received support from the Center for Mathematical Modeling through ANID/Basal projects \#ACE210010, \#FB210005, ANID/PIA \#AFB170001, and FONDECYT \#11220194.


\section{Introduction}
We study the Flat Vlasov-Poisson system with an attractive external potential. The main problem of this chapter is to search steady states solutions for this system as minimizers of Casimir-Energy functional, and study regularity and non-linear stability properties. This problem has been extensively studied (see \cite{REIN2000313,ISO-GUOREIN,https://doi.org/10.1007/s002200050674,StableSSDynamics}). In this context, as in the original case, we consider the Vlasov-Poisson system
\begin{equation}
    \label{eq:Vlasov_Poisson}
    \partial_t f + \bvec v\cdot\nabla_{\bvec x}f-\nabla_{\bvec x}U\cdot\nabla_{\bvec v} f = 0,
\end{equation}
but with the potential 
\begin{equation}
    U(t,\bvec x):= U_f(t,\bvec x) + U_{\rm ext}(\bvec x).
\end{equation}
Here $U_f$ corresponds to self-gravitational potential induced by $f$.
\begin{equation}
    \label{eq:Grav_pot}
    U_f(t,\bvec x):=-\int_{\R^2} \frac{\rho_f(t,\bvec y)}{|\bvec x-\bvec y|}\di \bvec y=-\int_{\R^2 \times \R^2} \frac{f(t,\bvec y,\bvec v)}{|\bvec x-\bvec y|}\di \bvec v \di \bvec y,
\end{equation}
where we denote $\rho=\rho_f$ the spatial density of $f$, and we write $U_f=U_{\rho}$. In particular, we study the case $U_{\rm ext}=U_{\rho_{\rm ext}}$ with $\rho_{\rm ext}\in L^1_{+}(\R^2)$. We emphasize that the potential formula in \eqref{eq:Grav_pot} differs from classic Vlasov-Poisson in $\R^2$, where the gravitational potential kernel behaives as $C\log|\bvec x|$. Flat Vlasov-Poisson system represents the idealized situation of a mass density confined to a plane in $\R^3$. Flat case with $\rho_{\rm ext}\equiv 0$ has been studied in \cite{FirtReinI} and a similar result in classic atractive Vlasov-Poisson system in dimension $3$ with $U_{\rm ext}(\bvec x)=-M/|\bvec x|$ was proved in \cite{schulze_2009}. The reduction idea comes from \cite{doi:10.1137/P0036141001389275}, and our aim is to prove that this reduction scheme still works in the case treated on this article.\\

Since we are interested in steady states solutions of \eqref{eq:Vlasov_Poisson}, both $f$ and $\rho_f$ are time independent flows, so henceforth we omit the time dependence in the notation. We define the Casimir-Energy Functional as
\begin{equation}
    \label{eq:Casimir_Energy}
    \mathcal{E}_\mathcal{C}(f)=E_{\rm kin}(f)+E_{\rm pot}(f)+\mathcal{C}(f),
\end{equation}
where the first term corresponds to kinetic energy
\begin{equation*}
    E_{\text{kin}}(f):=\int_{\R^d\times\R^d}\frac{|\bvec{v}|^2}{2}f(\bvec{x},\bvec{v})\di\bvec{v}\di\bvec{x}, 
\end{equation*}
and the second term is the potential energy associated to self-interaction and the central mass density respectively:
\begin{equation*}
    E_{\text{pot}}(f):=E_{\text{pot}}^1(f)+E_{\text{pot}}^{\varepsilon}(f),
\end{equation*}
where
\begin{equation*}
    E_{\text{pot}}(f):=-\frac{1}{2}\int_{\R^d\times\R^d}\frac{\rho_f(\bvec{x})\rho_f(\bvec{y})}{|\bvec{x}-\bvec{y}|}\di\bvec{x}\di\bvec{y}, 
\end{equation*}
and
\begin{equation*}
    E_{\text{pot}}^{\varepsilon}(f):=-\int_{\R^d\times\R^d}\frac{\rho_f(\bvec{x})\rho_\text{ext}(\bvec{y})}{|\bvec{x}-\bvec{y}|}\di\bvec{x}\di\bvec{y}.
\end{equation*}
The third term in \eqref{eq:Casimir_Energy} is usually called Casimir Functional, and it is defined as
\begin{equation}
    \label{eq:Casimir_Functional}
    \mathcal{C}(f)=\int_{\R^d\times\R^d}\Phi(f(\bvec x,\bvec v))\di\bvec v\di\bvec x
\end{equation}
for a suitable strictly convex function $\Phi\in C^1([0,\infty))$. It is easy to prove that steady states cannot minimize the classic energy of the system, expanding it around a steady state and noticing that it cannot be a critical point of the functional. The Casimir functional comes to overcome this obstacle. We define the following set of functions
\begin{equation}
    \label{eq:Feasible_set}
    \mathcal{F}_M:=\left\{f\in L^1_+(\R^4)\mid E_{\text{kin}}(f)+\mathcal{C}(f)<\infty,\int_{\R^d} f(\bvec x,\bvec v)\di\bvec{v}\di\bvec{x} = M\right\}.
\end{equation}
where $M>0$ is the mass of the system and is a constraint of the problem. Now we enunciate the main theorem proved in this paper.
\begin{theorem}
    \label{def:maintheorem1}
    Let $k\in(0,1)$ be a fixed parameter, and let $\Phi:[0,\infty)\rightarrow[0,\infty)$ be a strictly convex $C^1$ function such that $\Phi(0)=\Phi'(0)=0$, and $\Phi(x)\gtrsim x^{1+1/k}$ for large enough values of $x\geq 0$. Suppose that $\rho_{\text{ext}}\in L^{4/3}(\R^2)$ is stricly symmetric decreasing. Let $n=k+1\in (1,2)$. Then there exists some $E_0<0$ such that 
    \begin{equation}
        f_0=(\Phi')^{-1}(E_0-E)\chi_{E_0>E}
    \end{equation}
    is a minimizer of $\mathcal{E}_{\mathcal{C}}$ in $\mathcal{F}_M$. Moreover, if $\Phi\in C^{2}([0,\infty))$, $\Phi''>0$, and $\rho_{\text{ext}}\in L^{4/n}(\R^2)$, then $f_0$ is a stationary solution of Flat Vlasov-Poisson system with a central mass density. Here
    \begin{equation*}
        E(\bvec x,\bvec v)=\frac{1}{2}|\bvec v|^2+U_0(\bvec x)+U_{\rm ext}(\bvec x).
    \end{equation*}
    is the local energy per particle.
\end{theorem}
As in previous work (see \cite{doi:10.1137/P0036141001389275}, \cite{schulze_2009}, \cite{FirtReinI}) in the gravitational Vlasov-Poisson system, the aim is to construct solutions which are functions of the local energy per particle $E$ which minimize \eqref{eq:Casimir_Energy}, and this fact will be crucial to study the stability of these steady states. The problem with the Flat Vlasov-Poisson system is that even $E$ is not directly a steady state of the Vlasov equation. In our case, derivatives of gravitational potential terms induced by the solutions of the variational problem might not exist. Hence, the gravitational potential   
$U=U_0+U_{\rm ext}$ needs to be sufficiently regular in order to satisfy the equation as well.

\section{Reduced problem}
We solve a reduction of the original optimization problem. This reduction is defined over a suitable space of densities $\rho$, and we connect the solution for this reduced optimization problem with the original problem, constructing a solution using the Euler-Lagrange equation. Since we search steady state solutions of the system, we have that $f$ and $\rho_f$ are time independent flows. In order to construct a reduced problem, let
\begin{equation*}
    \mathcal{I}(g):=\int_{\R^2}\frac{|\bvec{v}|^2}{2}g(\bvec{v})+\Phi(g(\bvec{v}))\di\bvec v,
\end{equation*}
and define $\Psi:[0,\infty)\rightarrow[0,\infty)$ as follows: for each $r\geq 0$, we take the set of positive integrable functions $g$ such that $\mathcal{I}(g)<\infty$ and $g$ integrates $r$, i.e.
\begin{equation*}
    \mathcal{G}_r:=\left\{g\in L^1_+(\R^2)\mid \mathcal{I}(g)<\infty,\int g(\bvec v)\di\bvec v=r\right\}.
\end{equation*}
and define
\begin{equation}
    \label{eq:psi}
    \Psi(r):=\inf_{g\in\mathcal{G}_r}\mathcal{I}(g)
\end{equation}
The main goal is to minimize the Casimir-Energy functional over all functions $f(\bvec{x},\bvec{v})$ such that its spatial density is some fixed $\rho$, and after that, minimize it over $\rho$. Taking this into account, we set
\begin{equation}
    \mathcal{F}_M^r:=\left\{\rho\in L^{4/3}\cap L^1_+(\R^2)\mid\int\Psi(\rho(\bvec x))\di\bvec x<\infty,\int\rho(\bvec x)\di\bvec x=M\right\},
\end{equation}
and we define the reduced Casimir-Energy functional as follows
\begin{equation}
\mathcal{E}_{\mathcal{C}}^r(\rho):=\int\Psi(\rho(\bvec x))\di\bvec x+E_{\rm pot}(\rho)
\end{equation}
We state the following lemma, whose demonstration can be found in \cite{FirtReinI}. The main ideas to prove this lemma comes from the \textit{Legendre Transform} of a function $f$
\begin{lemma}
    \label{def:lema24}
    Let $\Phi\in C^1([0,+\infty))$, strictly convex, such that $\Phi(0)=\Phi'(0)=0$ and $\Psi$ be defined as in \eqref{eq:psi}, and extend both functions to $+\infty$ on $(-\infty,0)$. Then we have the following assertions: 
    \begin{enumerate}
        \item[(a)] $\Psi\in C^1([0,\infty))$, is strictly convex and $\Psi(0)=\Psi'(0)=0$. 
        \item[(b)] Let $k>0$ and $n=k+1$. Then
        \begin{enumerate}
            \item[(i)] If $\Phi(x)\simeq x^{1+1/k}$ for all $x\geq 0$, then $\Psi(x)\simeq x^{1+1/n}$ for $x\geq 0$
            \item[(ii)] If $\Phi(x)\gtrsim x^{1+1/k}$ for large enough values of $x\geq 0$, then $\Psi(x)\gtrsim x^{1+1/n}$ for large enough values of $x\geq 0$.
        \end{enumerate}
    \end{enumerate}
\end{lemma}
Previous lemma implies that $\Psi$ inherits properties from $\Phi$, so under the hypothesis of \ref{def:lema24} we can assume wlog that $\Psi\in C^1([0,\infty))$ is strictly convex, $\Psi(0)=\Psi'(0)=0$, and $\Psi(x)\gtrsim x^{1+1/n}$, with $n\in(1,2)$ fixed and large enough values of $x$. The following proposition will give sense to solve the reduced problem, connecting it with the original problem. The following proposition is a modified version of \cite[Teo 2.3]{FirtReinI} for the potential defined in \eqref{eq:Grav_pot}.

\begin{proposition}
\label{def:teo25}
Let $\rho_{ext}$ be the external spatial density. We have the following assertions: 
    \begin{enumerate}
        \item[(a)] For all functions $f\in\mathcal{F}_M$, we have that
        \begin{equation}
            \mathcal{E}_{\mathcal{C}}(f)\geq\mathcal{E}_{\mathcal{C}}^r(\rho_f)
            \label{def:th2.5a}
        \end{equation}
        with equality if $f$ is a minimizer of $\mathcal{E}_{\mathcal{C}}$ over $\mathcal{F}_M$.
        \item[(b)] Let $\rho_0$ be a minimizer of $\mathcal{E}_{\mathcal{C}}^r$ over $\mathcal{F}_M^r$ and let $U:=U_0+U_{\rm ext}$, where $U_0$ is the gravitational potential induced by $\rho_0$ and $U_{ext}$ is the gravitational potential induced by $\rho_{\rm ext}$. Suppose also that $\rho_0$ is spherically symmetric and nonincreasing. Then there exists a Lagrange multiplier $E_0<0$ such that almost everywhere we have
        \begin{equation}
            \rho_0=(\Psi')^{-1}(E_0-U)\chi_{E_0>U},
            \label{def:th2.5b1}
        \end{equation}
        and the function $f_0$ defined as 
        \begin{equation}
            f_0:=(\Phi')^{-1}(E_0-E)\chi_{E_0>E},
            \label{def:th2.5b2}
        \end{equation}
        where $E(\bvec x,\bvec v)=\frac{1}{2}|\bvec v|^2+U(\bvec x)$, is a minimizer of $\mathcal{E}_{\mathcal{C}}$ en $\mathcal{F}_M$.
    \end{enumerate}
\end{proposition}
\begin{proof}
    Proof of part (a) is quite similar as in \cite[Teo 2.3]{FirtReinI}. For part (b), let $\rho_0$ be a minimizer of $\mathcal{E}_\mathcal{C}^r$ in $\mathcal{F}_M^r$ and let
        $\varphi\in C^\infty_c(\R^2)$ such that $\text{supp}(\varphi)$ is strictly contained in $\text{supp}(\rho_0)$. If we define 
        \begin{equation*}
            \delta:=\inf_{\bvec x\in\text{supp}(\varphi)}\rho_0(\bvec x),
        \end{equation*}
        then $\delta>0$. As we have that $\rho_0$ is spherically symmetric and nonincreasing, then we can find $\lambda_\delta>0$ such that
        \begin{equation*}
            \lambda_\delta\cdot\displaystyle\sup_{\bvec x\in\text{supp}(\varphi)}|\varphi(\bvec x)|<\frac{\delta}{2}.
        \end{equation*}
        Thus we have that $\rho_0+\lambda\varphi\geq 0$ on $\text{supp}(\rho_0)$, for every $\lambda\in I_{\lambda_{\delta}}:=(-\lambda_\delta,\lambda_\delta)$. This allows us to construct a function $r:I_{\lambda_\delta}\rightarrow\R$ as
        \begin{equation*}
            \xi(\lambda)=\mathcal{E}_{\mathcal{C}}^r(\rho_0+\lambda\varphi),
        \end{equation*}
        which is well defined. Note that since $\rho_0$ is a minimizer of $\mathcal{E}_\mathcal{C}^r$, $\xi$ has a minimum in $\lambda=0$, therefore exists a Lagrange multiplier $E_0\in\R$ such that
        \begin{equation*}
            \xi'(0)=\int_{\R^2} E_0\cdot\varphi(\bvec x)\di\bvec x.
        \end{equation*}
        For the right-hand side, we compute
        \begin{equation*}
            \xi'(0)=\int_{\R^2}(\Psi'(\rho_0(\bvec x))+U(\bvec x))\varphi(\bvec x)\di\bvec x,
        \end{equation*}
        which implies that
        \begin{equation*}
            \int_{\R^2}(\Psi'(\rho_0(\bvec x))+U(\bvec x)-E_0)\varphi(\bvec x)\di\bvec x=0,
        \end{equation*}
        where $\varphi$ is an arbitrary element of $C_c^{\infty}(\R^2$) with $\text{supp}(\varphi)\subset\text{supp}(\rho_0)$. Hence we have almost everywhere on $\text{supp}(\rho_0)$ that
        \begin{equation*}
            \Psi'(\rho_0)=E_0-U,
        \end{equation*}
        and $E_0\leq U$ almost everywhere on $\R^2\setminus\text{supp}(\rho_0)$. Since $\Psi'$ is non negative and a bijective map, we have the Euler-Lagrange equation for the minimizer:
        \begin{equation*}
            \rho_0=(\Psi')^{-1}(E_0-U)\chi_{E_0>U},
        \end{equation*}
        where $U=U_0+U_{ext}$. On the other hand, we can prove that if $f_0$ is defined as in \ref{def:th2.5b2}, then $\rho_0=\rho_{f_0}$. Therefore, for an arbitrary $f\in\mathcal{F}_M$, we have that
        \begin{equation*}
            \mathcal{E}_\mathcal{C}(f)\geq\mathcal{E}_\mathcal{C}^r(\rho_f)\geq\mathcal{E}_\mathcal{C}^r(\rho_0)=\mathcal{E}_\mathcal{C}^r(\rho_{f_0})=\mathcal{E}_\mathcal{C}(f_0)
        \end{equation*}
        i.e.  $f_0$ minimizes $\mathcal{E}_\mathcal{C}$ over $\mathcal{F}_M$.  A straightforward calculation gives us that when $E_0>U$
        \begin{align*}
            \int_{\R^2} f_0(\bvec x,\bvec v)\di\bvec v&=\int_{E_0>E}(\Phi')^{-1}(E_0-E)\di\bvec v\\
            &=(\Psi^\ast)'(E_0-U)\\
            &=(\Psi')^{-1}(E_0-U)\\
            &=\rho_0,
        \end{align*}
        and both sides are zero where $E_0\leq U$, and since $U(\bvec x)\rightarrow 0$ when $|\bvec x|\rightarrow\infty$ we conclude that $E_0<0$. Here we denoted by $\Psi^\ast$ the Legendre transform of the function $\Psi$. 
\end{proof}

Proposition \ref{def:teo25} allows us to build a minimizer for the original problem through the reduced problem and writing the Euler-Lagrange equation for the reduction. Thus we just need to solve the variational reduced problem. The main idea to construct a solution resides in taking a minimizing sequence for the variational problem, and passing through subsequences, proving that it converges weakly in a suitable $L^p(\R^2)$ space. This procedure will give us a possible candidate for a minimizer.

\section{Variational approach: rearrangements}

In \cite{FirtReinI} and \cite{https://doi.org/10.1007/s002200050674}, the reduced problem for Flat Vlasov-Poisson system has been solved using a concentration compactness argument (see \cite{AIHPC_1984__1_2_109_0}) and the general ideas for this method can be reviewed in \cite{doi:10.1137/P0036141001389275}. In this paper we display a symmetrization argument, taking a minimizing sequence rearranged to obtain spherical symmetry, in order to keep the density concentrated in a finite region, avoiding the spatial translations and splittings (see \cite{lieb_loss_2001}). The following lemma explains why in the feasible set of the reduced problem, the value $p=4/3$ is useful for the following steps.

\begin{lemma}
    \label{def:lema28}
    If $\rho\in L^{4/3}(\R^2)$, then the gravitational potential $U_\rho$ is an element of $L^4(\R^2)$ which is the dual space of $L^{4/3}(\R^2)$. Moreover, the \textit{Coulomb energy} defined as
    \begin{equation}
        \label{def:coulomb}
        \mathcal{D}(\rho,\sigma):=\frac{1}{2}\iint\frac{\rho(\bvec x)\sigma(\bvec y)}{|\bvec x-\bvec y|}\di\bvec x\di\bvec y,
    \end{equation}
    is an inner product in $L^{4/3}(\R^2)$.
\end{lemma}

The next step is to prove that the reduced variational problem is well defined, in the sense that the Casimir-Energy functional is bounded below over the feasible set, and therefore the infimum does not \textit{``escape''}. We have the following lemma

\begin{lemma}
    \label{def:lema29}
    Under the assumptions over the function $\Psi$ given by \ref{def:lema24}, we have that $I_M>-\infty$.
\end{lemma}

\begin{proof}
    Let $\rho\in\mathcal{F}_M^r$. We have that
    \begin{equation*}
        \mathcal{E}_\mathcal{C}^r(\rho)=\int\Psi(\rho(\bvec{x}))\di\bvec{x}+E_{pot}^1(\rho)+E_{pot}^{\varepsilon}(\rho),
    \end{equation*}
    where $E^1_{pot}$ and $E^\varepsilon_{pot}$ correspond to potential energy associated to self-interaction and the external influence, respectively. We must prove this expression is bounded from below, uniformly in $\rho$. By the Hardy-Littlewood-Sobolev inequality we obtain the bounds
    \begin{equation}
        -E_{pot}^1(\rho)\lesssim\|\rho\|_{4/3}^2,\hspace{5mm} -E_{pot}^\varepsilon(\rho)\lesssim\|\rho\|_{4/3}.
    \end{equation}
    Furthermore, by the Riesz-Thorin Lemma, and considering that $\rho\in L^1(\R^2)$, we get the bound
    \begin{equation}
        \label{def:eq4}
        \|\rho\|_{4/3}\leq\|\rho\|_{1}^{\frac{3-n}{4}}\|\rho\|_{1+1/n}^{\frac{n+1}{4}}\lesssim\|\rho\|_{1+1/n}^{\frac{n+1}{4}}.   
    \end{equation}
    On the other hand, we have that there exists $\delta>0$ such that for every $\rho>\delta$, we have $\Psi(\rho)\geq C\rho^{1+1/n}$ for some suitable constant $C>0$. Hence, if we define $\{\rho>\delta\}:=\{\bvec x\in\R^2:\rho(\bvec{x})>\delta\}$, we have the following inequality
    \begin{align}
        \int_{\R^2}\rho(\bvec x)^{1+1/n}\di\bvec x&=\int_{\{\rho>\delta\}}\rho(\bvec x)^{1+1/n}\di\bvec x+\int_{\R^2\setminus \{\rho>\delta\}}\rho(\bvec x)^{1+1/n}\di\bvec x\\
        &\leq C\int_{\R^2}\Psi(\rho(\bvec x))\di\bvec x+\delta^{1/n}\int_{\R^2}\rho(\bvec x)\di\bvec x\\
        &\lesssim\int_{\R^2}\Psi(\rho(\bvec x))\di\bvec x+1.
        \label{eq:cota323}
    \end{align}
    Therefore, by the last bound we have
    \begin{equation}
        \|\rho\|_{1+1/n}^{\frac{n+1}{4}}=\left(\int_{\R^2}\rho(\bvec x)^{1+1/n}\di\bvec x\right)^{n/4}\lesssim\left(\int_{\R^2}\Psi(\rho(\bvec x))\di\bvec x\right)^{n/2}+1,
    \end{equation}
    and moreover
    \begin{equation}
        \|\rho\|_{1+1/n}^{\frac{n+1}{2}}=\left(\int_{\R^2}\rho(\bvec x)^{1+1/n}\di\bvec{x}\right)^{n/2}\lesssim\left(\int_{\R^2}\Psi(\rho(\bvec x))\di\bvec{x}\right)^{n/2}+1.
        \label{eq:cota1112}
    \end{equation}
    Combining the inequalities above, we have the following bounds for the potential energies
    \begin{equation}
        -E_{pot}^1(\rho)\lesssim\|\rho\|_{4/3}^2\lesssim\|\rho\|_{1+1/n}^{\frac{n+1}{2}}\lesssim\left(\int_{\R^2}\Psi(\rho(\bvec x))\di\bvec x\right)^{n/2}+1,
        \label{eqn:cota1618}
    \end{equation}
    \begin{equation}
        -E_{pot}^\varepsilon(\rho)\leq C\|\rho\|_{4/3}\lesssim\|\rho\|_{1+1/n}^{\frac{n+1}{4}}\lesssim\left(\int_{\R^2}\Psi(\rho(\bvec x))\di\bvec x\right)^{n/2}+1.
        \label{eqn:cota16180}
    \end{equation}
    Hence, we have the next inequality for the reduced energy functional
    \begin{equation}
        \label{def:cota1}
        \mathcal{E}_{\mathcal{C}}^r(\rho)\geq\int_{\R^2}\Psi(\rho(\bvec x))\di\bvec x -C\left(\int_{\R^2}\Psi(\rho(\bvec x))\di\bvec x\right)^{n/2}-C,
    \end{equation}
    where $C>0$ is a suitable constant. If we consider the function $g:[0,+\infty)\rightarrow\R$ defined as $g(\bvec{x})=-Cx^{n/2}+x-C$, since $0<n<2$, it is easy to verify that $g''(\bvec{x})\geq 0$ and the equation $g'(\bvec{x})=0$ has a solution, and therefore $g$ has a global minimum over $\R^+$. Hence, we have in \eqref{def:cota1} that
    \begin{equation*}
        \mathcal{E}_{\mathcal{C}}^r(\rho)=g\left(\int_{\R^2}\Psi(\rho(\bvec x))\di\bvec x\right)\geq\inf_{x\in\R}g(\bvec{x})>-\infty,
    \end{equation*}
    
    Thus we have inmediatly that $I_M>-\infty$ as we wanted to prove.
\end{proof}

A corollary of the lemma above is the next one, which allows us to find a minimizer candidate for the reduced variational problem.

\begin{corollary}
    \label{def:cor210}
    All rearranged minimizing sequences $(\rho_i)_{i\in\N}$ in $\mathcal{F}_M^r$ of $\mathcal{E}_{\mathcal{C}}^r$ are bounded in $L^{1+1/n}(\R^2)$ and $L^{4/3}(\R^2)$. In particular, each one of them has a subsequence that converges weakly in these spaces.
\end{corollary}

\begin{proof}
    Since $L^{1+1/n}(\R^2)$ and $L^{4/3}(\R^2)$ are reflexive spaces, the unit ball in both spaces is weak-sequentially compact, thus by the Banach-Alaoglu Theorem, it is sufficient to prove that every minimizing sequence is bounded in these spaces. Let $(\rho_i)_{i\in\N}$ be a minimizing sequence of reduced variational problem, i.e.
    \begin{equation*}
        \lim_{i\rightarrow\infty}\mathcal{E}_{\mathcal{C}}^r(\rho_i)=I_M.
    \end{equation*}
    In particular, $\mathcal{E}_{\mathcal{C}}^r(\rho_i)$ is bounded in $\R$. By \eqref{def:cota1}, we have that
    \begin{equation}
        \mathcal{E}_{\mathcal{C}}^r(\rho_i)\geq\int_{\R^2}\Psi(\rho_i(\bvec x))\di\bvec x\left(1-C\left(\int_{\R^2}\Psi(\rho_i(\bvec x))\di\bvec x\right)^{\frac{n}{2}-1}\right)-C,
        \label{def:cota2}
    \end{equation}
    for a suitable constant $C>0$. If $\int_{\R^2}\Psi(\rho_i)\di\bvec x$ is not bounded, as $0<n<2$, the right side of \eqref{def:cota2} tends to $+\infty$ as $i\rightarrow\infty$, contradicting that $\mathcal{E}_{\mathcal{C}}^r(\rho_i)$ is bounded. Hence $\Psi(\rho_i)$ is bounded in $L^1(\R^2)$, and since
    \begin{equation*}
        \int_{\R^2}\rho_i(\bvec x)^{1+1/n}\di\bvec x\lesssim 1+\int_{\R^2}\Psi(\rho_i(\bvec x))\di\bvec x,
    \end{equation*}
    the boundedness result in $L^{1+1/n}(\R^2)$ is direct, and as $\|\rho_i\|_{4/3}\lesssim\|\rho_i\|_{1+1/n}^{\frac{n+1}{4}}$ we have that the sequence is also bounded in $L^{4/3}(\R^2)$. In particular, as rearrangements preserve the norm, we have that $\|\rho_i\|_{1+1/n}=\|\rho_i^{\ast}\|_{1+1/n}$ and $\|\rho_i\|_{4/3}=\|\rho_i^{\ast}\|_{4/3}$, and therefore we conclude that the rearranged sequence is also bounded in each spaces and thus we can extract a weakly convergent subsequence from the rearranged of initial sequence which converges in each spaces.\\ 
    
    Now, we just need to prove that the weak limit is the same. For this, let $\rho_0$ and $\rho_1$ the weak limits of $(\rho_i)_{i\in\N}$ in $L^{1+1/n}(\R^2)$ and $L^{4/3}(\R^2)$ respectively, and let $A$ some arbitrary Lebesgue measurable set. Then we have that $\rho_0\chi_{A\cap B(0,R)}$ and $\rho_1\chi_{A\cap B(0,R)}$ converge pointwise to $\rho_0\chi_{A}$ and $\rho_1\chi_{A}$ as $R\rightarrow\infty$, respectively, and both are dominated by $\rho_0$ and $\rho_1$ which are integrable. Since $\chi_{A\cap B(0,R)}$ is an element of every $L^{p}(\R^2)$ with $1\leq p\leq\infty$, by weak convergence we have that if $i\rightarrow\infty$, then
    \begin{equation*}
        \int\rho_i(\bvec{x})\chi_{A\cap B(0,R)}(\bvec{x})\di\bvec{x}\longrightarrow \int\rho_0(\bvec{x})\chi_{A\cap B(0,R)}(\bvec{x})\di\bvec{x}=\int\rho_1\chi_{A\cap B(0,R)}(\bvec{x})\di\bvec{x}
    \end{equation*}
    and taking $R\rightarrow\infty$, by the Dominated Convergence Theorem we have that
    \begin{equation*}
        \int\rho_0(\bvec{x})\chi_{A}(\bvec{x})\di\bvec{x}=\int\rho_1(\bvec{x})\chi_{A}(\bvec{x})\di\bvec{x},
    \end{equation*}
    for every Lebesgue measurable set $A$, and therefore $\rho_0-\rho_1=0$ a.e. and thus
    \begin{equation*}
        \|\rho_0-\rho_1\|_{1+1/n}=0.
    \end{equation*}
\end{proof}
The following lemma establishes that every rearranged minimizing sequence still remains this property. This is crucial for our approach.
\begin{lemma}
    \label{def:lema211}
    If $(\rho_i)_{i\in\R}$ is a minimizing sequence of $\mathcal{E}_{\mathcal{C}}^r$ in $\mathcal{F}_M^r$, then the rearranged sequence $(\rho_i^\ast)_{i\in\R}$ is also a minimizing sequence.
\end{lemma}

\begin{proof}
    Since the rearrangement preserves the norm, we have
    \begin{equation}
        \int\rho_i^{\ast}(\bvec{x})\di\bvec{x}=M,
        \label{def:eq2}
    \end{equation}
    and as $\Psi$ is nonnegative and convex such that $\Psi(0)=0$, by the nonnexpansivity of rearrangment, we have that
    \begin{equation}
        \int\Psi(\rho_i^{*}(\bvec{x}))\di\bvec{x}\leq\int\Psi(\rho_i(\bvec{x}))\di\bvec{x}<\infty.
        \label{def:iq2}
    \end{equation}
    By \eqref{def:eq2} and \eqref{def:iq2} we have that $\rho_i^{\ast}\in\mathcal{F}_M^r$, for all $i\in\N$, and thus
    \begin{equation*}
        I_M\leq\mathcal{E}_{\mathcal{C}}^r(\rho_i^{\ast}).
    \end{equation*}
    By the Riesz Rearrangement Inequality, we have the bound
    \begin{equation*}
        E_{pot}^1(\rho_i)=-\frac{1}{2}\iint\frac{\rho_i(\bvec{x})\rho_i(\bvec{y})}{|\bvec{x}-\bvec{y}|}\di\bvec{x}\di\bvec{y}\geq-\frac{1}{2}\iint\frac{\rho_i^{\ast}(\bvec{x})\rho_i^{\ast}(\bvec{y})}{|\bvec{x}-\bvec{y}|}\di\bvec{x}\di\bvec{y}=E_{pot}^{1}(\rho_i^\ast).
    \end{equation*}
    Moreover, as $\rho_{ext}$ is strictly symmetric decreasing, by the simplest rearrangement inequality and as $\|\rho_{ext}\|_2=\|\rho^\ast_{ext}\|_2$, we have that $\rho_{ext}=\rho_{ext}^\ast$ and therefore
    \begin{equation*}
        -E_{pot}^\varepsilon(\rho_i)=\iint\frac{\rho_{ext}(\bvec{x})\rho_i(\bvec{y})}{|\bvec{x}-\bvec{y}|}d\bvec{x}d\bvec{y}=\int -U_{ext}(\bvec{x})\rho_i^{\ast}(\bvec{x})d\bvec{x}=-E_{pot}^\varepsilon(\rho_i^\ast),
    \end{equation*}
    and hence
    \begin{equation*}
        E_{pot}^{\varepsilon}(\rho_i^{\ast})+ E_{pot}^1(\rho_i^{\ast})\leq E_{pot}^{\varepsilon}(\rho_i)+E_{pot}^1(\rho_i).
    \end{equation*}
    which implies that
    \begin{equation*}
        \lim_{i\rightarrow\infty}\mathcal{E}_\mathcal{C}(\rho_i^{\ast})=I_M
    \end{equation*}
    as desired.
\end{proof}

\begin{remark}
    Based on Corollary \ref{def:cor210} and Lemma \ref{def:lema211}, we assume without loss of generality that the minimizing sequence also is spherically symmetric and nonincreasing. As $(\rho_i)_{i\in\N}$ is weakly-sequentially compact, passing through subsequences we can assume from now that there is some $\rho_0$ such that $\rho_i\rightharpoonup\rho_0$ on $L^{1+1/n}(\R^2)$ and $L^{4/3}(\R^2)$. 
\end{remark}

Although the weak convergence usually might not give us enough information about the weak limit $\rho_0$, in this case we have the following result, which establishes that the weak convergence implies strong convergence of potential energies.

\begin{lemma}
    \label{def:lema212}
    Let $(\rho_i)_{i\in\N}$ be the rearranged minimizing sequence. Then 
    \begin{equation}
        E_{pot}(\rho_i)\rightarrow E_{pot}(\rho_0),
    \end{equation}
    where $\rho_0$ is the weak limit in $L^{1+1/n}(\R^2)$.
\end{lemma}

\begin{proof}
    Since $U_{ext}=U_{\rho_{ext}}$ with $\rho_{ext}\in L^{4/3}(\R^2)$, by Lemma \ref{def:lema28} we have that $U_{ext}\in L^{4}(\R^2)=(L^{4/3}(\R^2))^{\ast}$ and by weak convergence in $L^{4/3}(\R^2)$ (see \ref{def:cor210}) we have that $E_{pot}^{\varepsilon}(\rho_i)\rightarrow E_{pot}^{\varepsilon}(\rho_0)$. It is sufficient to show that
    \begin{equation}
        E_{pot}^1(\rho_i)\rightarrow E_{pot}^1(\rho_0).
    \end{equation}
    Define $\sigma_i:=\rho_i-\rho_0\rightharpoonup 0$ in $L^{1+1/n}(\R^2)$ and note that if $\mathcal{D}$ is the Coulomb energy defined in \ref{def:coulomb}, then
    \begin{align*}
        E_{pot}^1(\sigma_i)=E_{pot}^1(\rho_i)-E_{pot}^1(\rho_0)-2\mathcal{D}(\sigma_i,\rho_0).
    \end{align*}
    By weak convergence we have that $\mathcal{D}(\sigma_i,\rho_0)\rightarrow 0$, it is enough to prove that $E_{pot}^1(\sigma_i)\rightarrow 0$. Let $R_1>0$ and split the integral in $E_{pot}(\sigma_i)$ in two parts, one of them inside, and the other outside of the strip $|x-y|<R_1$
    \begin{align*}
        \iint\frac{\sigma_i(\bvec x)\sigma_i(\bvec y)}{|\bvec x-\bvec y|}\di\bvec x\di\bvec y&=\iint_{|\bvec x-\bvec y|<R_1}\frac{\sigma_i(\bvec x)\sigma_i(\bvec y)}{|\bvec x-\bvec y|}\di\bvec x\di\bvec y\\
        &+\iint_{|\bvec x-\bvec y|\geq R_1}\frac{\sigma_i(\bvec x)\sigma_i(\bvec y)}{|\bvec x-\bvec y|}\di\bvec x\di\bvec y.
    \end{align*}
    We denote by these two integrals $I_1$, $I_2$, and we will try to find small bounds for them. First, as $\sigma_i\in L^{1+1/n}(\R^2)$ and $1/|\cdot|\in L^{(n+1)/2}(B(0,R_1))$, using Hölder's inequality, Young's inequality, and the fact that $(\sigma_i)_{i\in\N}$ is bounded in $L^{1+1/n}(\R^2)$, we bound the first integral as follows
    \begin{equation*}
        I_1\lesssim\left\|\frac{\chi_{B(0,R_1)}}{|\cdot|}\right\|_{(n+1)/2}.
    \end{equation*}
    Therefore, we note that
    \begin{equation*}
        \left\|\frac{\chi_{B(0,R_1)}}{|\cdot|}\right\|_{(n+1)/2}=\left(\int^{2\pi}_{0}\int_{0}^{R_1}\frac{1}{r^{(n+1)/2}}r\di r\di\theta\right)^{2/(n+1)}=CR_1^{(3-n)/2},
    \end{equation*}
    for a suitable constant $C>0$. Here the last term vanishes when $R_1\rightarrow 0$ since $0<n<2$, and thus for $R_1$ small enough, $I_1<\epsilon$, for an arbitrary $\epsilon>0$. For the second integral, we define 
    \begin{equation*}
        U_{R_2}=\{(\bvec x,\bvec y)\in\R^4:|\bvec x|\geq R_2\vee|\bvec y|\geq R_2\}.
    \end{equation*}
    and we split the second integral inside and outside the set $U_{R_2}$ 
    \begin{equation*}
        I_2=\iint_{|\bvec x-\bvec y|\geq R_1\cap U_{R_2}}\frac{\sigma_i(\bvec x)\sigma_i(\bvec y)}{|\bvec x-\bvec y|}\di\bvec x\di\bvec y+\iint_{|\bvec x-\bvec y|\geq R_1\cap U_{R_2}^c}\frac{\sigma_i(\bvec x)\sigma_i(\bvec y)}{|\bvec x-\bvec y|}\di\bvec x\di\bvec y.
    \end{equation*}
    If we call these two integrals by $I_{2,1}$ and $I_{2,2}$, respectively, by Hardy-Littlewood-Sobolev we have the bound 
    \begin{equation*}
        |I_{2,1}|\lesssim\|\sigma_i\chi_{B(0,R_2)^c}\|_{4/3}\|\sigma_i\|_{4/3}.
    \end{equation*}
    As in the proof of Lemma \ref{def:lema29}, we have that $\|\sigma_i\|_{4/3}\lesssim\|\sigma_i\|_{1+1/n}^{\frac{n+1}{4}}$ and as $\sigma_i$ converges weakly to $0$ in $L^{1+1/n}(\R^2)$ and thus is bounded in that space. Then this inequality implies that $\sigma_i$ is also bounded in $L^{4/3}(\R^2)$. Hence, by Minkowski's inequality we have
    \begin{equation*}
        |I_{2,1}|\lesssim\|\rho_i\chi_{B(0,R_2)^c}\|_{4/3}+\|\rho_0\chi_{B(0,R_2)^c}\|_{4/3}.
    \end{equation*}
    As $\rho_i$ is spherically symmetric and nonincreasing, pointwise on $\bvec x\in\R^2$, the function $\rho_i$ is dominated by the average over a ball centered in the origin and radius $|\bvec x|$, i.e. 
    \begin{equation*}
        \rho_i(\bvec x)\leq\frac{1}{|B(0,1)|\cdot |\bvec x|^2}\int_{B(0,|\bvec x|)}\rho_i(y)\di\bvec y\lesssim\frac{1}{|\bvec x|^2}.
    \end{equation*}
    Therefore
    \begin{align*}
        \|\rho_i\chi_{B(0,R_2)^c}\|_{4/3}&\lesssim\left(\int_{B(0,R_2)^c}\frac{1}{|\bvec x|^{8/3}}\di\bvec x\right)^{3/4}\lesssim R_2^{-\frac{1}{2}}.
    \end{align*}
    On the other hand, in the same way as before
    \begin{equation*}
        \|\rho_0\chi_{B(0,R_2)^c}\|_{4/3}\lesssim\|\rho_0\chi_{B(0,R_2)^c}\|_{1+1/n}^{\frac{n+1}{4}},    
    \end{equation*}
    and as $\rho_i\rightharpoonup\rho_0$ in $L^{1+1/n}(\R^2)$, then by the weakly lower semicontinuity of the norm, we have that
    \begin{equation*}
        \|\rho_0\chi_{B(0,R_2)^c}\|_{1+1/n}\leq\liminf_{i\rightarrow\infty}\|\rho_i\chi_{B(0,R_2)^c}\|_{1+1/n}.
    \end{equation*}
    A similar brief calculation gives us the bound
    \begin{align*}
        \|\rho_i\chi_{B(0,R_2)^c}\|_{1+1/n}\lesssim R_2^{-\frac{2}{n+1}},
    \end{align*}
    and therefore
    \begin{equation*}
        \|\rho_0\chi_{B(0,R_2)^c}\|_{4/3}\lesssim\|\rho_0\chi_{B(0,R_2)^c}\|_{1+1/n}^{\frac{n+1}{4}}\lesssim R_2^{-\frac{1}{2}}.
    \end{equation*}
    Hence, we have that
    \begin{equation*}
        |I_{2,1}|\lesssim R_2^{-\frac{1}{2}}\rightarrow 0,
    \end{equation*}
    when $R_2\rightarrow\infty$. Finally, we have that
    \begin{align*}
        I_{2,2}&=\iint_{|\bvec x-\bvec y|\geq R_1}\frac{\sigma_i(\bvec x)\sigma_i(\bvec y)}{|\bvec x-\bvec y|}\chi_{B(0,R_2)}(\bvec x)\chi_{B(0,R_2)}(y)\di\bvec x\di\bvec y\\
        &=\int\sigma_i(\bvec x)h_i(\bvec x)\di\bvec x,
    \end{align*}
    where we defined the function $h_i$ as follows 
    \begin{equation*}
        h_i(\bvec x):=\chi_{B(0,R_2)}(\bvec x)\int_{|\bvec x-\bvec y|\geq R_1}\frac{\sigma_i(\bvec y)}{|\bvec x-\bvec y|}\chi_{B(0,R_2)}(\bvec y)\di\bvec y.
    \end{equation*}
    Now, for every $\bvec x\in\R^2$ we denote by $\varphi_{\bvec x}$ the function defined as
    \begin{equation*}
        \varphi_{\bvec x}(y):=\frac{\chi_{B(0,R_2)}(\bvec y)}{|\bvec x-\bvec y|}\chi_{\R^2\setminus B(0,R_1)}(\bvec x-\bvec y),
    \end{equation*}
    in order to write
    \begin{equation*}
        h_i(\bvec x)=\chi_{B(0,R_2)}(\bvec x)\int\sigma_i(\bvec y)\varphi_{\bvec x}(\bvec y)\di\bvec y.
    \end{equation*}
    We prove that $\varphi_{\bvec x}\in L^{n+1}(\R^2)$. Indeed, we have that 
    \begin{equation*}
        \|\varphi_{\bvec x}\|_{n+1}^{n+1}=\int\frac{\chi_{B(0,R_2)}(\bvec x)}{|\bvec x-\bvec y|^{n+1}}\chi_{\R^2\setminus B(0,R_1)}(\bvec x-\bvec y)\di\bvec y\leq\frac{1}{R_1^{n+1}}\mathcal|B(0,R_2)|<\infty.
    \end{equation*}
    Thus, as $\sigma_i\rightharpoonup 0$ in $L^{1+1/n}(\R^2)$, we have that $h_i(\bvec x)\rightarrow 0$, for all $\bvec x\in\R^2$, and as $|h_i(\bvec x)|\lesssim\chi_{B(0,R_2)}(\bvec x)\in L^{n+1}(\R^2)$, then by the Dominated Convergence Theorem, we have that $h_i\rightarrow 0$ en $L^{n+1}(\R^2)$. Then, by the Hölder's inequality we have that 
    \begin{equation*}
        I_{2,2}\leq\|\sigma_i\|_{1+1/n}\|h_i\|_{n+1}\lesssim\|h_i\|_{n+1}<\epsilon,
    \end{equation*}
    for all large enough values of $i$. Therefore, when $R_2\rightarrow\infty$, for small enoguh values of $R_1$ and large enough values of $i$, we have 
    \begin{equation*}
        \iint\frac{\sigma_i(\bvec x)\sigma_i(\bvec y)}{|\bvec x-\bvec y|}\di\bvec x\di\bvec y=I_1+I_{2,1}+I_{2,2}<2\epsilon,
    \end{equation*}
    which implies that $E_{pot}^1(\sigma_i)\rightarrow 0$. Thus
    \begin{equation*}
        E_{pot}^1(\rho_i)-E_{pot}^1(\rho_0)=E_{pot}^1(\sigma_i)+2\mathcal{D}(\sigma_i,\rho_0)\rightarrow 0,
    \end{equation*}
    as desired.\par
\end{proof}
Convergence of potential energies proved above, gives us a strong evidence that the weak limit $\rho_0$ in $L^{1+1/n}(\R^2)$ of the rearranged minimizing sequence is a good candidate to be a minimizer for the functional $\mathcal{E}_{\mathcal{C}}^r$. We just need to prove that $\mathcal{E}_\mathcal{C}^r(\rho_0)$ is almost the infimum over the feasible space and also that $\rho_0$ is an element of this set. We are ready to prove the first part of Theorem \ref{def:maintheorem1}.

\begin{proof}[Proof of Theorem \ref{def:maintheorem1}, first part]
Let $(\rho_i)_{i\in\N}\subseteq\mathcal{F}_M^r$ be the rearranged minimizing sequence for $\mathcal{E}_{\mathcal{C}}^r$. By Corollary \ref{def:cor210}, there exists $\rho_0$ such that $\rho_i\rightharpoonup\rho_0$ in $L^{1+1/n}(\R^2)$. For $R>0$ we have that $\chi_{B(0,R)}\in L^{n+1}(\R^2)$, therefore by weak convergence we have that
    \begin{equation*}
        M=\int\rho_i(\bvec{x})\di\bvec{x}\geq\int_{B(0,R)}\rho_i(\bvec{x})\di\bvec{x}\longrightarrow\int_{B(0,R)}\rho_0(\bvec{x})\di\bvec{x}
    \end{equation*}
    thus, when $R\rightarrow\infty$, we have that
    \begin{equation*}
        \int\rho_0(\bvec{x})\di\bvec{x}\leq M,
    \end{equation*}
     and interpolating in the same way as in \eqref{def:eq4}, and along with the above, we can prove that $\rho_0\in L^1_+(\R^2)\cap L^{4/3}(\R^2)$, and by Lemma \ref{def:lema212} we have that $E_{pot}(\rho_i)\rightarrow E_{pot}(\rho_0)$. Using Mazur's Lemma, we can find a family of nonnegative and finite sequences $(B_n)_{n\in\N}$, where $B_n=(\alpha_k^n)_{k=n}^{N_n}$ such that 
    \begin{equation*}
        \sum_{k=n}^{N_n}\alpha_k^n=1,
    \end{equation*}
    and the sequence $(\hat{\rho}_n)_{n\in\N}$ defined as 
    \begin{equation*}
        \hat{\rho}_n=\sum_{k=n}^{N_n}\alpha_k^n\rho_k,
    \end{equation*}
    converges strongly to $\rho_0$ in $L^{1+1/n}(\R^2)$. Thus, there is a subsequence  $(\hat{\rho}_{n_j})_{j\in\N}$ which converges pointwise almost everywhere to $\rho_0$. Since the map $\rho\mapsto\Psi(\rho)$ is strictly convex, we have that the map $\rho\mapsto\int\Psi(\rho)d\bvec{x}$ is convex, and therefore
    \begin{equation*}
        \int\Psi(\hat{\rho}_{n_j})\di\bvec{x}=\sum_{k=n_j}^{N_{n_j}}\alpha_k^{n_j}\int\Psi(\rho_k)\di\bvec{x}\leq\sup_{k\geq n_j}\int\Psi(\rho_k)\di\bvec{x}.
    \end{equation*}
    By Fatou's Lemma and using the fact that $\Psi\in C^1([0,\infty))$, we have that
    \begin{align*}
        \int\Psi(\rho_0)\di\bvec{x}=\int\liminf_{j\rightarrow\infty}{\Psi(\hat{\rho}_{n_j})}\di\bvec{x}\leq \limsup_{k\rightarrow\infty}{\int\Psi(\rho_{k})}\di\bvec{x}.
    \end{align*}
    As $E_{pot}(\rho_0)=\displaystyle\limsup_{i\rightarrow\infty}{E_{pot}(\rho_i)}$, and $(\rho_i)_{i\in\N}$ is a minimizing sequence of $\mathcal{E}_{\mathcal{C}}^r$ in $\mathcal{F}_M^r$, we have that
    \begin{equation*}
        \mathcal{E}_{\mathcal{C}}^r(\rho_0)\leq\limsup_{i\rightarrow\infty}{\int\Psi(\rho_{i})}\di\bvec{x}+\limsup_{i\rightarrow\infty}{E_{pot}(\rho_i)}=I_M.
    \end{equation*}
    Hence $\mathcal{E}_{\mathcal{C}}^r(\rho_0)\leq I_M$ and thus we only must prove that $\rho_0$ is an element of the feasible set, i.e. $\rho_0\in\mathcal{F}_M^r$. For this, if $\int\Psi(\rho_0(\bvec{x}))\di\bvec{x}$ is not bounded, the inequality in \eqref{def:cota1} implies that
    \begin{equation*}
        I_M=\infty,
    \end{equation*}
    which is a contradiction. Thus, it is enough to prove that
    \begin{equation*}
        \int\rho_0(\bvec{x})\di\bvec{x}=M.    
    \end{equation*}
    We will proceed by contradiction. Assume otherwise that
    \begin{equation*}
       0< M''=\int\rho_0(\bvec{x})\di\bvec{x}<M,
    \end{equation*}
    and let's consider $\Bar{\rho}_R:=\rho_0\chi_{B(0,R)}$, where $R>0$ is such that 
    \begin{equation}
        \label{def:eq3}
        M':=\int\Bar{\rho}_R(\bvec{x})\di\bvec{x}<M''<M.
    \end{equation}
    Define $\mathcal{E}_{\mathcal{C}}^{r,0}$ and $I_M^0$ as the Casimir-Energy reduced functional and the infimum of the minimization reduced problem respectively, in the case of Flat Vlasov-Poisson without external influence. Then
    \begin{equation*}
        I_M^{0}:=\inf_{\rho\in\mathcal{F}_M^r}\mathcal{E}_{\mathcal{C}}^{r,0}(\rho),
    \end{equation*}
    where
    \begin{equation*}
        \mathcal{E}_{\mathcal{C}}^{r,0}(\rho)=\int\Psi(\rho(\bvec{x}))\di\bvec{x}+E_{pot}^1(\rho)
    \end{equation*}
    In \cite[Lema 3.4. (a)]{FirtReinI} it was proved that $I_M^{0}<0$ for every $M>0$, and the assertion \cite[Lema 3.4. (b)]{FirtReinI} of this lemma implies that the map $M\mapsto I_M^0$ is nonincreasing in $M$. Since $\Bar{\rho}_R(\bvec{x})\rightarrow\rho_0(\bvec{x})$ pointwise when $R\rightarrow\infty$ and is a function dominated by $\rho_0\in L^{1+1/n}(\R^2)$, the Dominated Convergence Theorem implies that $\Bar{\rho}_R\rightarrow\rho_0$ in $L^{1+1/n}(\R^2)$ and in particular $\Bar{\rho}_R\rightharpoonup\rho_0$ in this space. As $\Psi$ is strictly convex and $\Psi'$ is a bijection over $[0,\infty)$, then is strictly increasing and thus
    \begin{equation*}
        \int\Psi(\Bar{\rho}_R(\bvec{x}))\di\bvec{x}\leq\int\Psi(\rho_0(\bvec{x}))\di\bvec{x}.
    \end{equation*}
    In a similar way as before, if we rearrange $\Bar{\rho}_R$, as $\Psi$ is convex and $\Psi(0)=0$, by nonexpansivity of rearrangements together with the inequality above we have that
    \begin{equation*}
        \int\Psi(\Bar{\rho}^\ast_R(\bvec{x}))\di\bvec{x}\leq\int\Psi(\rho_0(\bvec{x}))\di\bvec{x}.
    \end{equation*}
    Hence, by Lemma \ref{def:lema212}, we have that $E_{pot}(\Bar{\rho}_R^\ast)\rightarrow E_{pot}(\rho_0)$. Let $\epsilon>0$ such that
    \begin{equation}
        \label{eq:c1}
        I_{M-M''}^{0}<-\epsilon
    \end{equation}
    and take $R$ large enough such that $E_{pot}(\Bar{\rho}_R^\ast)-E_{pot}(\rho_0)\leq\frac{\varepsilon}{2}$. Then
    \begin{align}
        \mathcal{E}_{\mathcal{C}}^r(\Bar{\rho}_R^\ast)=\mathcal{E}_{\mathcal{C}}^r(\rho_0)+\frac{\epsilon}{2}.\label{eq:c2}
    \end{align}
    By \eqref{def:eq3}, we have that $\delta:=M-M'\geq M-M''>0$ and we can take $\varphi\in\mathcal{F}_{\delta}^r$ such that $supp(\varphi)\subseteq B(0,R')$ with $R'>0$ and
    \begin{equation}
        \label{eq:c3}
        \mathcal{E}_{\mathcal{C}}^{r,0}(\varphi)<\frac{I_\delta^{0}}{2}\leq\frac{I_{M-M''}^{0}}{2}.
    \end{equation}
    In this way, if $a\in\R^2$ is such that $|a|=R+R'$, then
    \begin{equation}
        \label{eq:c4}
        \int\Psi(\Bar{\rho}_R^\ast(\bvec{x})+\varphi(\bvec{x}-a))\di\bvec{x}=\int\Psi(\Bar{\rho}_R^\ast(\bvec{x}))\di\bvec{x}+\int\Psi(\varphi(\bvec{x}))\di\bvec{x}.
    \end{equation}
    Otherwise
    \begin{equation}
        \label{eq:c5}
        -\mathcal{D}(\Bar{\rho}_R^\ast+\varphi(\cdot-a),\Bar{\rho}_R^\ast+\varphi(\cdot-a))\leq-\mathcal{D}(\Bar{\rho}_R^\ast,\Bar{\rho}_R^\ast)-\mathcal{D}(\varphi,\varphi),
    \end{equation}
    and
    \begin{equation}
        \label{eq:c6}
        -\mathcal{D}(\rho_{ext},\Bar{\rho}_R^\ast+\varphi(\cdot-a))\leq-\mathcal{D}(\rho_{ext},\Bar{\rho}_R^\ast).
    \end{equation}
    Joining all the expressions in \eqref{eq:c1}, \eqref{eq:c2}, \eqref{eq:c3}, \eqref{eq:c4}, \eqref{eq:c5} and \eqref{eq:c6}, we are able to bound as follows
    \begin{align*}
        \mathcal{E}_{\mathcal{C}}^r(\Bar{\rho}_R^\ast+\varphi(\cdot-a))&\leq\mathcal{E}_{\mathcal{C}}^r(\Bar{\rho}_R^\ast)+\mathcal{E}_{\mathcal{C}}^{r,0}(\varphi)\\
        &<\mathcal{E}_{\mathcal{C}}^r(\rho_0)+\frac{\epsilon}{2}+\frac{I_{M-M''}^0}{2}\\
        &<\mathcal{E}_{\mathcal{C}}^r(\rho_0)\\
        &=I_{M}.
    \end{align*}
    As the rearrangement preserves the $L^1(\R^2)$ norm, then $\Bar{\rho}_R^\ast\in\mathcal{F}_{M'}^r$ and $\varphi\in\mathcal{F}_{\delta}^r$, then $\Bar{\rho}_R^\ast+\varphi(\cdot-a)\in\mathcal{F}_M^r$, which is a contradiction. Hence
    \begin{equation*}
        \int\rho_0(\bvec{x})\di\bvec{x}=M,   
    \end{equation*}
     and therefore $\rho_0\in\mathcal{F}_M^r$, which implies that $\rho_0$ is a minimizer of $\mathcal{E}_{\mathcal{C}}^r$ over the feasible set, i.e. a solution of the reduced variational problem. Hence, by Theorem \ref{def:teo25}, we have that the function $f_0$ which satisfies the following identity
    \begin{equation*}
        f_0=(\Phi')^{-1}(E_0-E)\chi_{E_0>E},
    \end{equation*}
    is a minimizer of the Casimir-Energy functional.
\end{proof}

\begin{remark} 

Note that the rearrangement argument also works for the original case in Flat Vlasov-Poisson system, setting $\rho_{ext}\equiv0\in L^{4/3}(\R^2)$. As in the proof of Lemma \ref{def:lema211}, we can prove that if $\rho_0$ is the solution of the reduced variational problem, the symmetric rearrangement of $\rho_0$ is also a solution. Hence, we can assume that $\rho_0$ is spherically symmetric and nonincreasing. When we proved that $\rho_0\in\mathcal{F}_M^r$, we assumed that $\rho_0$ is zero almost everywhere, but this assertion is true. Indeed, as $I_M^0<0$, then there exist $\rho\in\mathcal{F}_M^r$ such that $\mathcal{E}_\mathcal{C}^{r,0}(\rho)<0$, then
\begin{equation*}
    \mathcal{E}_\mathcal{C}^r(\rho_0)\leq\mathcal{E}_\mathcal{C}^{r,0}(\rho)+E_{pot}^\varepsilon(\rho)<0,
\end{equation*}
thus $\rho_0$ cannot be zero almost everywhere. 
\end{remark}

\section{Regularity of gravitational potential}

As it was mentioned along the text, the idea is to construct stationary solutions which are functions of the energy $E$. The problem in Flat Vlasov-Poisson system is that even $E$ is not directly a steady state of the Vlasov equation. That is, in order to $f=E$ be a solution of 
\begin{equation}
    \label{eq:vlasovsdystate}
    \bvec{v}\cdot\nabla_{\bvec{x}} f - \nabla_{\bvec{x}} U\cdot\nabla_{\bvec{v}} f = 0,
\end{equation}
$U=U_0+U_{ext}$ needs to be sufficiently regular to make sense on the second term in the left side in \eqref{eq:vlasovsdystate}. This is obviously true in Vlasov-Poisson in $\R^3$ considering that in this case we are looking for solutions in which the gravitational potential is a solution of Poisson's equation $\Delta U=4\pi\rho$, so there exist at least two derivatives of $U$. In the flat case the regularity cannot be assure easily, because the potential is not constructed as solutions of Poisson's equation in dimension $d=2$, but it is imposed as the convolution in $\R^2$ of the spatial density and $-1/|\cdot|$. In this section this issue is aborded and solved. First of all, we have the following result.

\begin{lemma}
    \label{def:contpot}
    Assume that $\rho_{ext}\in L^{q}(\R^2)$, where $q=\frac{p}{p-1}$, $1<p<2$. Let $\rho_0$ the solution of the reduced variational problem. Then $U_{0}$ and therefore $U$ are continuous.
\end{lemma}

\begin{proof}
    Since $\rho_0,\rho_{ext}\in L^{4/3}(\R^2)$, $U=U_0+U_{ext}\in L^4(\R^2)$, and by the equation $\Psi'(\rho_0)=E_0-U$ a.e. on $\text{supp}(\rho_0)=\{E_0-U>0\}$ we conclude that $\Psi'(\rho_0)\in L^4(\R^2)$. Recall that 
    \begin{equation*}
        \Psi(x)\gtrsim x^{1+1/n},
    \end{equation*}
    for large values of $x\geq 0$. Thus, by the convexity of $\Psi$ we have that
    \begin{equation*}
        \Psi'(\rho)\gtrsim\rho^{1/n},
    \end{equation*}
    where the last inequality holds for all $\rho\geq\delta$, for some $\delta\geq 0$ fixed. Thus we have that
    \begin{align*}
        \int\rho_0(\bvec{x})^{4/n}\di\bvec{x}&=\int_{\{\rho_0\geq\delta\}}\rho_0(\bvec{x})^{4/n}\di\bvec{x}+\int_{\{\rho_0<\delta\}}\rho_0(\bvec{x})^{4/n}\di\bvec{x}\\
        &\lesssim\int_{\{\rho_0\geq\delta\}}\Psi'(\rho_0(\bvec{x}))^4\di\bvec{x}+\int_{\{\rho_0<\delta\}}\rho_0(\bvec{x})^{4/n}\di\bvec{x}\\
        &\lesssim\int\Psi'(\rho_0(\bvec{x}))^4\di\bvec{x}+1,
    \end{align*}
    and therefore $\rho_0\in L^{4/n}(\R^2)$. It is easy to prove that if $n\in(1,2)$, then $1/|\cdot|\chi_{B(0,R)}\in L^{4/n}(\R^2)^\ast=L^{\frac{4}{4-n}}(\R^2)$ and $1/|\cdot|\chi_{B(0,R)^c}\in L^{4/3}(\R^2)^\ast= L^4(\R^2)$ for every $R>0$. Hence
    \begin{equation*}
        -U_0=\frac{1}{|\cdot|}\ast\rho_0=\frac{1}{|\cdot|}\chi_{B(0,R)}\ast\rho_0+\frac{1}{|\cdot|}\chi_{B(0,R)^c}\ast\rho_0,
    \end{equation*}
    is continuous. Finally, as $\rho_{ext}\in L^{4/3}(\R^2)\cap L^q(\R^2)$, we have that
    \begin{equation*}
        -U_{ext}=\frac{1}{|\cdot|}\ast\rho_{ext}=\frac{1}{|\cdot|}\chi_{B(0,R)}\ast\rho_{ext}+\frac{1}{|\cdot|}\chi_{B(0,R)^c}\ast\rho_{ext},
    \end{equation*}
    and the argument is the same as above, just note that $1/|\cdot|\chi_{B(0,R)}\in L^{p}(\R^2)$.
\end{proof}

It is clear that the result proved above is still insufficient. On the other hand, for the solution of the variational problem
\begin{equation*}
    f_0=(\Phi')^{-1}(E_0-E)\chi_{E_0>E},
\end{equation*}
we also need more regularity for $\Phi$. In particular, we need to differentiate the inverse of the first derivative. For this, it is enough to have $\Phi\in C^2([0,\infty))$ and $\Phi''>0$, and thus, if we could prove more regularity for the potential $U$, then $f_0$ would be a steady state solution of the Flat Vlasov-Poisson with the external potential. Another way to investigate the regularity of the potential comes from to the study of its Fourier Transform $\mathcal{F}$. Note that if $d=2$, then
\begin{equation}
    \mathcal{F}(U)\simeq\mathcal{F}(\rho_{0}+\rho_\varepsilon)\cdot\frac{1}{|\cdot|}.
\end{equation}
Since $\rho_0\in L^{4/n}(\R^2)\cap L^1(\R^2)$, interpolating we conclude that $\rho_0\in L^2(\R^2)$ and therefore by Plancherel's Theorem, we have that $\mathcal{F}(\rho_0)\in L^2(\R^2)$. If $\rho_{ext}\in L^{2}(\R^2)$, as before we have $\mathcal{F}(\rho_{ext})$ and hence $\mathcal{F}(\rho_{0,\varepsilon})\in L^2(\R^2)$. Then we conclude that the map $|\cdot|\mathcal{F}(U)\in L^2(\R^2)$, and in particular we have that $(1+|\cdot|)\mathcal{F}(U)\in L^2_{loc}(\R^2)$, which implies that $U\in H^1_{loc}(\R^2)$. This is not enough to embed this space to reach more regularity. The main idea to overcome this problem comes from the Riesz Transform (see \cite{SPS_1988__22__485_0}), and more general, from the following result.

\begin{theorem}
    \label{def:teo117}
    The operator $T$ defined as
    \begin{equation}
        Tf:=-\frac{1}{|\cdot|}\ast f,
    \end{equation}
    maps elements of $W^{k,p}(\R^2)$ into $W^{k+1,p}(\R^2)$.
\end{theorem}

\begin{proof}
    If $f\in W^{k,p}(\R^2)$, then we need to prove that $D^\alpha(Tf)\in L^p(\R^2)$ for every multiindex $\alpha$ with $|\alpha|=k+1$. Note that $D^\alpha(Tf)=D^\beta(\partial_j(Tf))$ where $\beta$ is a multiindex with $|\beta|=k$ and $D^\beta f\in L^p(\R^2)$. Then
    \begin{equation*}
        \mathcal{F}(R_j(D^\beta f))(\bvec{x})=-i\frac{x_j}{|x|}\mathcal{F}(D^\beta f)(\bvec{x}).
    \end{equation*}
    Thus, if we denote by $R_j$ the Riesz Transform in the coordinate $j\in\{1,2\}$, we have
    \begin{align*}
        \mathcal{F}(R_j(D^\beta f))(\bvec{x})\simeq\mathcal{F}(D^{\beta}(\partial_j Tf))(\bvec{x}).
    \end{align*}
    And then
    \begin{equation*}
        \|D^\alpha(Tf)\|_p=\|D^{\beta}(\partial_j Tf)\|_p\lesssim\|D^\beta f\|_p,
    \end{equation*}
    which implies the desires result.
\end{proof}

For our main problem, we need the following lemma.

\begin{lemma}
    \label{def:contpsi}
    If $\Phi\in C^2([0,\infty))$ and $\Phi''>0$ in $[0,\infty)$, then $(\Psi')^{-1}\in C^1([0,\infty))$.
\end{lemma}

\begin{proof}
    Note that since $\Phi(r)=+\infty$ on $(-\infty,0)$, for $\lambda>0$
    \begin{equation*}
        \Psi^\ast(\lambda)=\int_{|\bvec{v}|<\sqrt{2\lambda}}\Phi^\ast\left(\lambda-\frac{|\bvec{v}|^2}{2}\right)\di\bvec{v}.
    \end{equation*}
    Recall that
    \begin{equation*}
        \frac{d}{d\lambda}[(\Psi')^{-1}](\lambda)=\frac{d^2}{d\lambda^2}[\Psi^\ast](\lambda)=\frac{d^2}{d\lambda^2}\int_{|\bvec{v}|<\sqrt{2\lambda}}\Phi^\ast\left(\lambda-\frac{|\bvec{v}|^2}{2}\right)\di\bvec{v}.
    \end{equation*}
    Since
    \begin{align*}
        \frac{d}{d\lambda}\int_{|\bvec{v}|<\sqrt{2\lambda}}\Phi^\ast\left(\lambda-\frac{|\bvec{v}|^2}{2}\right)\di\bvec{v}\simeq\int^{\sqrt{2\lambda}}_0(\Phi')^{-1}\left(\lambda-\frac{r^2}{2}\right)r\di r,
    \end{align*}
    we have
    \begin{align*}
        \frac{d}{d\lambda}[(\Psi')^{-1}](\lambda)\simeq\int_0^{\sqrt{2\lambda}}\frac{r\di r}{\Phi''\left((\Phi')^{-1}\left(\lambda-\frac{r^2}{2}\right)\right)},
    \end{align*}
    
    which implies the desired result.
\end{proof}

Next we prove the following theorem, which allows us to say that $f_0$ (from \eqref{def:th2.5b2}) inside on the support of $\rho_0$ is locally an steady state solution of Flat Vlasov-Poisson with a central mass density.

\begin{theorem}
    \label{teo:119}
    Suppose that $\Phi\in C^2([0,\infty))$, $\Phi''>0$, and $\rho_{ext}\in L^{4/n}(\R^2)$. Then $U$ belongs to the Hölder space $C^{1,1-\frac{n}{2}}(\Omega)$, for every bounded set $\Omega\subset\R^2$ with $C^1$ boundary, and therefore
    \begin{equation*}
        f_0=(\Phi')^{-1}(E_0-E)\chi_{E_0>E},
    \end{equation*}
    is an steady state solution of Flat Vlasov-Poisson system with a central mass density in $\Omega$.
\end{theorem}

\begin{proof}
    By Theorem \ref{def:teo117}, since $\rho_{0,ext}=\rho_0+\rho_{ext}\in L^{4/n}(\R^2)=W^{0,4/n}(\R^2)$, we have
    \begin{equation*}
        T\rho_{0,ext}=U\in W^{1,4/n}(\R^2).
    \end{equation*}
    Denote by $\nabla$ the weak gradient. Thus from \eqref{def:th2.5b1} we have
    \begin{equation*}
        \nabla\rho_0=-((\Psi')^{-1})'((E_0-U)_+)\nabla U_+.
    \end{equation*}
    By Lemma \ref{def:contpot} and Lemma \ref{def:contpsi}, $((\Psi')^{-1})'((E_0-U)_+)$ is bounded a.e., and therefore we have $\rho_0\in W^{1,4/n}(\R^2)$. Again by Theorem \ref{def:teo117} we can conclude that $U\in W^{2,4/n}(\R^2)$. Hence, by the Sobolev Embedding Theorem we conclude that
    \begin{equation*}
        U\in C^{1,1-\frac{n}{2}}(\bar{\Omega}). 
    \end{equation*}
    for every bounded set $\Omega\subset\R^2$ with $C^1$ boundary. Provided of the regularity of $U$ in every bounded set $\Omega$ of $S$ with $C^1$ boundary, and the regularity of $\Phi$, $f_0$ is an steady state solution of the Vlasov equation in $\Omega$, which is the required result.
\end{proof}

\begin{remark}
    We observe that since $\rho_0$ is continuous, spherically symmetric and nonincreasing, the support of $\rho_0$ is a closed ball $\overline{B}(0,R_0)$ for some $R_0>0$. Then we can take $\Omega=B(0,R)\subset\R^2$ with $R>R_0$, where Theorem \ref{teo:119} holds. Outside of $B(0,R)\subset\R^2\setminus\text{supp}(\rho_0)$ the function $1/|\cdot|$ is $C^1$, and therefore $U$ preserves this regularity. Hence, we can replace open and bounded sets in the hypotesis of Theorem \ref{teo:119} with the whole plane $\R^2$. This result completes the demonstration of Theorem \ref{def:maintheorem1}. 
\end{remark}

\section{Stability of minimizers}

The existence of a minimizer for the variational problem has just been proved in the last section. In the same way as in \cite{FirtReinI,schulze_2009}, we prove a similar result of stability, now for the flat case with central mass density. Before that, we prove some useful results. We expand over the minimizer $f_0$ given by \ref{def:maintheorem1}, and we have that

\begin{equation}
    \mathcal{E}_{\mathcal{C}}(f)-\mathcal{E}_{\mathcal{C}}(f_0)=d(f,f_0)-E_{pot}^1(\rho_f-\rho_{f_0}),
\end{equation}
where
\begin{equation}
    d(f,f_0):=\iint\left[\Phi(f)-\Phi(f_0)+E(f-f_0)\right]\di\bvec{v}\di\bvec{x},
\end{equation}
and in this case
\begin{equation*}
    E(\bvec{x},\bvec{v})=\frac{1}{2}|v|^2+U(\bvec{x})=\frac{1}{2}|v|^2+U_0(\bvec{x})+U_{ext}(\bvec{x})
\end{equation*}
is the energy defined in \ref{def:teo25}. Thus, as $\Phi$ is strictly convex, we can show that $d(f,f_0)\geq0$, where the equality holds if and only if $f=f_0$. We have the following lemma. 

\begin{lemma}
    \label{def:lema213}
    Let $(f_i)_{i\in\N}$ a minimizing sequence for $\mathcal{E}_\mathcal{C}$ en $\mathcal{F}_M$. Then $\rho_{f_i}$ is a minimizing sequence for $\mathcal{E}_\mathcal{C}^r$ in $\mathcal{F}_M^r$.
\end{lemma}

\begin{proof}
    It is clear from \ref{def:teo25} that
    \begin{equation*}
        \mathcal{E}_{\mathcal{C}}(f_i)\geq\mathcal{E}_{\mathcal{C}}^r(\rho_{f_i})\geq\inf_{\rho\in\mathcal{F}_M^r}\mathcal{E}_{\mathcal{C}}^r(\rho).
    \end{equation*}
    If $\rho_0$ is the minimizer for $\mathcal{E}_\mathcal{C}^r$ obtained from reduced problem and $f_0$ is the minimizer for $\mathcal{E}_\mathcal{C}$ induced by $\rho_0$,
    as $\rho_0=\rho_{f_0}$ we have that $\mathcal{E}_{\mathcal{C}}^r(\rho_{f_0})=\mathcal{E}_{\mathcal{C}}^r(\rho_0)$. By Proposition \ref{def:teo25}, we have that
    \begin{equation*}
        \mathcal{E}_{\mathcal{C}}(f_i)\rightarrow\inf_{f\in\mathcal{F}_M}\mathcal{E}_{\mathcal{C}}(f)=\mathcal{E}_{\mathcal{C}}(f_0)=\mathcal{E}_{\mathcal{C}}^r(\rho_0)=\inf_{\rho\in\mathcal{F}_M^r}\mathcal{E}_{\mathcal{C}}^r(\rho),
    \end{equation*}
    as we wanted to prove.
\end{proof}

\begin{lemma}
    \label{def:lema214}
    Let $(f_i)_{i\in\N}$ a minimizing sequence for $\mathcal{E}_{\mathcal{C}}$ in $\mathcal{F}_M$. Then it is bounded in $L^{1+1/k}(\R^4)$. In particular the sequence is weakly-sequentially compact in that space.
\end{lemma}

\begin{proof}
    In a similar way as \eqref{eq:cota323}, using the fact that $f^{1+1/k}\lesssim\Phi(f)$ for large enough values of $f$, we can prove that
    \begin{align*}
        \iint f_i(\bvec{x},\bvec{v})^{1+1/k}\di\bvec{v}\di\bvec{x}\lesssim\mathcal{C}(f_i)+1. 
    \end{align*}
    where we recall that $\mathcal{C}$ is the Casimir-Energy functional defined in \eqref{eq:Casimir_Functional}: 
    \begin{equation*}
       \mathcal{C}(f_i)=\iint\Phi(f_i(\bvec{x},\bvec{v}))\di\bvec{v}\di\bvec{x}.
    \end{equation*}
    Similar as in the proof of the bounds \eqref{eqn:cota1618} and \eqref{eqn:cota16180} in  \ref{def:cor210}, we can prove that
    \begin{equation*}
        -E_{pot}(f_i)=-E_{pot}(\rho_{f_i})\lesssim\mathcal{C}(f_i)^{k/2}+1,
    \end{equation*}
    and hence we obtain the bound
    \begin{align*}
       \mathcal{E}_{\mathcal{C}}(f_i)\geq\mathcal{C}(f_i)\left(1-C\cdot\mathcal{C}(f_i)^{\frac{k}{2}-1}\right)-C,
    \end{align*}
    for a suitable constant $C>0$. Therefore, if the Casimir functional $\mathcal{C}(f_i)$ is not bounded, then $\mathcal{E}_\mathcal{C}(f_i)\rightarrow+\infty$, which is a contradiction. Hence $\mathcal{C}(f_i)$ is bounded, and therefore, so is $(f_i)_{i\in\N}$ in $L^{1+1/k}(\R^4)$. As $L^{1+1/k}(\R^4)$ is a reflexive space, by Banach-Aloglu Theorem, we can find a subsequence which is weakly-sequentially compact in that space. 
\end{proof}

We know that by Lemma \ref{def:lema213} proved above, if $(f_i)_{i\in\N}$ is a minimizing sequence for $\mathcal{E}_\mathcal{C}$ in $\mathcal{F}_M$, then $\rho_{f_i}$ is minimizing sequence for $\mathcal{E}_{\mathcal{C}}^r$ in $\mathcal{F}_M^r$, and passing through subsequence, we have already saw that it converges weakly to a minimizer $\tilde{\rho}_0$ for the reduced functional. We have the following result:

\begin{lemma}
    \label{def:lema215}
    Let $(f_i)_{i\in\N}$ a minimizing sequence for $\mathcal{E}_{\mathcal{C}}$ in $\mathcal{F}_M$ and let $f_0$ the minimizer obtained from Theorem \ref{def:maintheorem1}, induced by $\rho_0$, and we suppose that is unique. Then passing through subsequence, we have that $\rho_{f_i}\rightharpoonup\rho_0$ in $L^{1+1/n}(\R^2)$, and passing through subsequence, $f_i\rightharpoonup f_0$ en $L^{1+1/k}(\R^4)$.
\end{lemma}

\begin{proof}
    By Lemma \ref{def:lema213}, we know that $(\rho_{f_i})_{i\in\N}$ is a minimizing sequence for $\mathcal{E}_{\mathcal{C}}^r$ in $\mathcal{F}_M^r$, and by Corollary \ref{def:cor210} and uniqueness of $f_0$, passing through subsequence we conclude that $\rho_{f_i}\rightharpoonup\rho_0$ in $L^{1+1/n}(\R^2)$, where $\rho_0=\rho_{f_0}$. By Lemma \ref{def:lema214}, we have that, passing by subsequence again, $f_i\rightharpoonup\tilde{f}_0$ en $L^{1+1/k}(\R^4)$. We will prove that $\rho_0=\rho_{\tilde{f}_0}$ almost everywhere. For them, let $A$ be an arbitraty Lebesgue measurable set and let $R_1,R_2>0$ be arbitrary positive numbers. We have that
    \begin{align*}
        \int_{A}\chi_{B(0,R_1)}(\bvec{x})\rho_{\tilde{f}_0}(\bvec{x})\di\bvec{x} &=\iint_{A}\chi_{B(0,R_1)}(\bvec{x})\chi_{B(0,R_2)}(\bvec{v})\tilde{f}_0(\bvec{x},\bvec{v})\di\bvec{v}\di\bvec{x}\\
        &+\iint_{A}\chi_{B(0,R_1)}(\bvec{x})\chi_{B(0,R_2)^c}(\bvec{v})\tilde{f}_0(\bvec{x},\bvec{v})\di\bvec{v}\di\bvec{x}.
    \end{align*}
    It is obvious that the map $(\bvec{x},\bvec{v})\mapsto\chi_{B(0,R_1)}(\bvec{x})\chi_{B(0,R_2)}(\bvec{v})$ is an element of the dual space of $L^{1+1/k}(\R^4)$ which is $L^{k+1}(\R^4)$, and therefore
    \begin{align*}
        \iint_{A}\chi_{B(0,R_1)}(\bvec{x})\chi_{B(0,R_2)}(\bvec{v})\tilde{f}_0(\bvec{x},\bvec{v})\di\bvec{v}\di\bvec{x}&\leq \lim_{i\rightarrow\infty}\iint_{A}\chi_{B(0,R_1)}(\bvec{x})f_i(\bvec{x},\bvec{v})\di\bvec{v}\di\bvec{x}\\
        &=\lim_{i\rightarrow\infty}\int_{A}\chi_{B(0,R_1)}(\bvec{x})\rho_{f_i}(\bvec{x})\di\bvec{x}\\
        &=\int_{A}\chi_{B(0,R_1)}(\bvec{x})\rho_{0}(\bvec{x})\di\bvec{x},
    \end{align*}
    where the last equality comes from the fact that $\bvec{x}\mapsto\chi_{B(0,R_1)}(\bvec{x})$ is an element of $L^{1+1/n}(\R^2)^{\ast}=L^{n+1}(\R^2)$. On the other hand, we have that
    \begin{align*}
        \iint_{A}\chi_{B(0,R_1)}(\bvec{x})\chi_{B(0,R_2)^c}(\bvec{v})\tilde{f}_0(\bvec{x},\bvec{v})\di\bvec{v}\di\bvec{x}&\leq\frac{2}{R_2^2}\iint\frac{|\bvec{v}|^2}{2}f_0(\bvec{x},\bvec{v})\di\bvec{v}\di\bvec{x}\\
        &=\frac{2}{R_2^2}E_{kin}(\tilde{f}_0).
    \end{align*}
    Therefore
    \begin{equation*}
        \int_{A}\chi_{B(0,R_1)}(\bvec{x})\rho_{\tilde{f}_0}(\bvec{x})\di\bvec{x}\leq\int_{A}\chi_{B(0,R_1)}(\bvec{x})\rho_{0}(\bvec{x})\di\bvec{x}+\frac{2}{R_2^2}E_{kin}(\tilde{f}_0),
    \end{equation*}
    and thus if $R_1,R_2\rightarrow+\infty$ we have found
    \begin{equation}
        \int_A\rho_{\tilde{f}_0}(\bvec{x})\di\bvec{x}\leq\int_A\rho_0(\bvec{x})\di\bvec{x}.
    \end{equation}
    On the other hand, by weak convergence again we have that
    \begin{align*}
        \int_A\chi_{B(0,R_1)}(\bvec{x})\rho_0(\bvec{x})\di\bvec{x}&=\liminf_{i\rightarrow\infty}\iint_A\chi_{B(0,R_1)}(\bvec{x})\chi_{B(0,R_2)}(\bvec{v})f_i(\bvec{x},\bvec{v})\di\bvec{v}\di\bvec{x}+\\
        &\liminf_{i\rightarrow\infty}\iint_A\chi_{B(0,R_1)}(\bvec{x})\chi_{B(0,R_2)^c}(\bvec{v})f_i(\bvec{x},\bvec{v})\di\bvec{v}\di\bvec{x}.
    \end{align*}
    Therefore, we have
    \begin{align*}
        \liminf_{i\rightarrow\infty}\iint_A\chi_{B(0,R_1)}(\bvec{x})\chi_{B(0,R_2)}(\bvec{v})f_i(\bvec{x},\bvec{v})\di\bvec{v}\di\bvec{x}&\leq\int_A\chi_{B(0,R_1)}(\bvec{x})\rho_{\tilde{f}_0}(\bvec{x})\di\bvec{x}.
    \end{align*}
    It is enough to prove that the second term goes to 0. Note that
    \begin{align*}
        \liminf_{i\rightarrow\infty}\iint_A\chi_{B(0,R_1)}(\bvec{x})\chi_{B(0,R_2)^c}(v)f_i(\bvec{x},\bvec{v})\di\bvec{v}\di\bvec{x}&\leq\frac{2}{R_2^2}\liminf_{i\rightarrow\infty}E_{kin}(f_i)\\
        &\lesssim\frac{1}{R_2^2}.
    \end{align*}
    Thus, if $R_1,R_2\rightarrow\infty$, we have that
    \begin{equation}
        \int_A\rho_0(\bvec{x})\di\bvec{x}\leq\int_A\rho_{\tilde{f}_0}(\bvec{x})\di\bvec{x},
    \end{equation}
    and therefore
    \begin{equation*}
        \int_A\rho_0(\bvec{x})\di\bvec{x}=\int_A\rho_{\tilde{f}_0}(\bvec{x})\di\bvec{x},
    \end{equation*}
    and this is for every Lebesgue measurable set $A$. Hence $\rho_0=\rho_{\tilde{f}_0}$ almost everywhere and therefore $E_{pot}(\tilde{f}_0)=E_{pot}(f_0)$. Thus
    \begin{align*}
        \mathcal{E}_{\mathcal{C}}(\tilde{f}_0)\leq\inf_{f\in\mathcal{F}_M}\mathcal{E}_{\mathcal{C}}(f).
    \end{align*}
    
    Since $\rho_{\tilde{f}_0}=\rho_0$ almost everywhere, we have that $\tilde{f}_0$ integrates $M$, and thus it is an element of the feasible set $\mathcal{F}_M$. Therefore is a minimizer of Casimir-Energy functional, and by the uniqueness of minimizer we have that $\tilde{f}_0=f_0$, as we desired to prove.
\end{proof} 

Next we prove the main result from this section, which gives us a notion of stability for the minimizer given by Theorem \ref{def:maintheorem1}, analogous to the stability result from \cite{FirtReinI}. 

\begin{theorem}
\label{def:teo215}
Let $f_0$ a minimizer for $\mathcal{E}_{\mathcal{C}}$ in $\mathcal{F}_{M}$ and we suppose that is unique, and let $\rho_0:=\rho_{f_0}$. Let $\varepsilon>0$, then there is some $\delta>0$ such that for every solution of flat Vlasov-Poisson system with central mass density $t\mapsto f(t)$, with initial datum $f(0)\in C_c^{1}(\R^4)\cap\mathcal{F}_M$, if
\begin{equation}
    \label{def:eq111}
    d(f(0),f_0)-E_{pot}^1(\rho_{f(0)}-\rho_0)<\delta,
\end{equation}
then
\begin{equation}
    \label{def:eq112}
    d(f(t),f_0)-E_{pot}^1(\rho_{f(t)}-\rho_0)<\varepsilon,
\end{equation}
for every $t\geq 0$.

\end{theorem}

\begin{proof}
    We will proceed by contradiction. If the assertion is not true, then there exists some $\epsilon_0>0$, such that for every $i\in\N$, there is some $t_i>0$, and a solution $f_i$ from flat Vlasov-Poisson system with central mass density, such that $f_i(0)\in C^1_c(\R^4)\cap\mathcal{F}_M$ with
    \begin{equation}
        \label{def:eq5}
        d(f_i(0),f_0)-E_{pot}^1(\rho_{f_i(0)}-\rho_0)<\frac{1}{i},
    \end{equation}
    and
    \begin{equation}
    \label{def:eq6}
    d(f_i(t_i),f_0)-E_{pot}^1(\rho_{f_i(t_i)}-\rho_0)\geq\epsilon_0.
\end{equation}
By \eqref{def:eq5}, we have that $\mathcal{E}_{\mathcal{C}}(f_i(0))-\mathcal{E}_{\mathcal{C}}(f_0)=d(f_i(0),f_0)-E_{pot}^1(\rho_{f_i(0)}-\rho_0)\rightarrow 0$. As the Casimir-Energy functional $\mathcal{E}_{\mathcal{C}}$ is a conserved quantity, then $\mathcal{E}_{\mathcal{C}}(f_i(0))=\mathcal{E}_{\mathcal{C}}(f_i(t_i))\rightarrow\mathcal{E}_{\mathcal{C}}(f_0)$. Then we have that $(f_i(t_i))_{i\in\N}$ is a minimizing sequence for $\mathcal{E}_{\mathcal{C}}$ in $\mathcal{F}_M$, and therefore passing through subsequence, we have that $(\rho_{f_i(t_i)})_{i\in\N}$ is a minimizing sequence for $\mathcal{E}_{\mathcal{C}}^r$ in $\mathcal{F}_M^r$, and as $f_0$ is unique, by Lemma \ref{def:lema214}, we have that $\rho_{f_i(t_i)}\rightharpoonup\rho_0$ in $L^{1+1/n}(\R^2)$. By Corollary \ref{def:cor210}, we have that $E_{pot}^1(\rho_{f_i(t_i)}-\rho_0)\rightarrow 0$, and thus as
\begin{equation*}
    \mathcal{E}_{\mathcal{C}}(f_i(t_i))-\mathcal{E}_{\mathcal{C}}(f_0)=d(f_i(t_i),f_0)-E_{pot}^1(\rho_{f_i(t_i)}-\rho_0),
\end{equation*}
we have that $d(f_i(t_i),f_0)\rightarrow 0$, which contradicts \ref{def:eq6}.

\end{proof}

\begin{remark}
    As we mentioned in \ref{def:coulomb}, the Coulomb energy $\mathcal{D}$ is an inner product over $L^{4/3}(\R^2)$ which induces a norm in that space, given by
    \begin{equation}
        \|\rho\|_{pot}:=\mathcal{D}(\rho,\rho)^{1/2}=(-E_{pot}^1(\rho))^{1/2},
    \end{equation}
    and therefore we can replace $-E_{pot}(\cdot)$ by $\|\cdot\|_{pot}$ in Theorem \ref{def:teo215}.
\end{remark}

\begin{corollary}
    \label{def:cor217}
    Let $\varepsilon>0$. Under the assumptions from \ref{def:teo215}, and supposing that $\|f(0)\|_{1+1/k}=\|f_0\|_{1+1/k}$, then there is some $\delta>0$ such that if \eqref{def:eq111} holds, then
    
    \begin{equation}
        \|f(t)-f_0\|_{1+1/k}<\varepsilon.
    \end{equation}
\end{corollary}

\begin{proof}
    In the same way as in the proof from the above theorem, if we assume the opposite, we can build a minimizing sequence $(f_i(t_i))_{i\in\N}$ such that 
    \begin{equation*}
        \|f_i(t_i)\|_{1+1/k}=\|f_i(0)\|_{1+1/k}=\|f_0\|_{1+1/k}.
    \end{equation*}
    By Lemma \ref{def:lema215}, we have that $f_i(t_i)\rightharpoonup f_0$ and also $\|f_i(t_i)\|_{1+1/k}\rightarrow\|f_0\|_{1+1/k}$. Using the fact that $L^{1+1/k}(\R^4)$ is uniformly convex, we have that $f_i(t_i)$ converges to $f_0$ in $L^{1+1/k}(\R^4)$, which contradicts $\|f_i(t_i)-f_0\|_{1+1/k}\geq\epsilon_0$.
\end{proof}

\begin{remark} 
Interpolating we have that the result of Corollary \ref{def:cor217} is true in $L^p(\R^4)$ norm, for every $1<p\leq 1+1/k$. An open and interesting question is to study the uniqueness of the minimizer for the variational problem, which is an important hypothesis in the non-linear stability properties proved above. Global existence of classical solutions of Flat Vlasov-Poisson system given an initial datum $C^1_c(\R^2)$ have not been proved yet. The results of non-linear stability are conditional to have the existence of these kind of solutions.
\end{remark}

\bibliographystyle{alpha}
\bibliography{bibliografia}
\nocite{*}

\end{document}